\documentclass[12pt]{article}

\usepackage{amssymb}%
\usepackage{amsmath}%
\usepackage{amsthm}%

\usepackage{xcolor}%

\allowdisplaybreaks%

{\theoremstyle{plain}%
 \newtheorem{theorem}{Theorem}

}
{\theoremstyle{remark}

}
{\theoremstyle{definition}
\newtheorem{definition}{Definition}
\newtheorem{example}{Example}
}

\begin{document}

\begin{center}
 {\Large The prime-counting Copeland--Erd\H{o}s constant} 
\end{center}

\begin{center}
{\textsc{John M. Campbell}} 

 \ 

\end{center}

\begin{abstract}
 Let $(a(n) : n \in \mathbb{N})$ denote a sequence of nonnegative integers. Let $0.a(1)a(2)...$ denote the real number obtained by concatenating the 
 digit expansions, in a fixed base, of consecutive entries of $(a(n) : n \in \mathbb{N})$. Research on digit expansions of this form has mainly to do with 
 the normality of $0.a(1)a(2)...$ for a given base. Famously, the Copeland--Erd\H{o}s constant $0.2357111317...$, for the case whereby $a(n)$ equals 
 the $n^{\text{th}}$ prime number $p_{n}$, is normal in base 10. However, it seems that the ``inverse'' construction given by concatenating the decimal 
 digits of $(\pi(n) : n \in \mathbb{N})$, where $\pi$ denotes the prime-counting function, has not previously been considered. Exploring the distribution of 
 sequences of digits in this new constant $0.0122...9101011...$ would be comparatively difficult, since the number of times a fixed $m \in \mathbb{N} $ 
 appears in $(\pi(n) : n \in \mathbb{N})$ is equal to the prime gap $g_{m} = p_{m+1} - p_{m}$, with the behaviour of prime gaps notoriously elusive. Using 
 a combinatorial method due to Sz\"{u}sz and Volkmann, we prove that Cram\'{e}r's conjecture on prime gaps implies the normality of $0.a(1)a(2)...$ in a 
 given base $g \geq 2$, for $a(n) = \pi(n)$. 
\end{abstract}



\section{Introduction}
 The study of normal numbers forms a large and important area in probabilistic number theory. Much of our notation concerning normal numbers is 
 based on the work of Sz\"{u}sz and Volkmann in \cite{SzuszVolkmann1994}. Following \cite{SzuszVolkmann1994}, we let $g \geq 2$ be a fixed 
 parameter throughout our article, letting it be understood that we are working with digits in base $g$, unless otherwise specified. For a real value $ 
 \alpha$, and for a block $E$ of digits, let $A_{E}(\alpha, n)$ denote the number of copies of $E$ within the first $n$ digits of $\alpha$. The real number 
 $\alpha$ is said to be \emph{normal of order $k$} if 
\begin{equation}\label{SVnormal}
 \lim_{n \to \infty} \frac{ A_{E}(\alpha, n) }{n} = \frac{1}{g^{\ell(E)}} 
\end{equation}
 for all $E$ such that $\ell(E) = k$, where $\ell(E)$ denotes the \emph{length} of $E$, or the number of digits in $E$, counting multiplicities. For $k = 1$, the 
 specified property concerning \eqref{SVnormal} is referred to as \emph{simple normality}. A real number $\alpha$ is said to be \emph{normal} if it is 
 normal for all orders $k \in \mathbb{N}$. In this article, we introduce a constant related to the prime-counting function that may be seen as something 
 of an inverse relative to the construction of the famous Copeland--Erd\H{o}s constant \cite{CopelandErdos1946}, and we prove, under the assumption 
 of Cram\'{e}r's conjecture, that this new constant is normal. This is inspired by past work related to the normality of real numbers defined via 
 concatenations of number-theoretic functions, as in 
 \cite{CattCoonsVelich2016,DeKoninckKatai2016,DeKoninckKatai2011,PollackVandehey2015Besicovitch,PollackVandehey2015Some,Vandehey2013}. 

\subsection{Background}
 For a sequence 
 $(a(n) : n \in \mathbb{N})$ of nonnegative integers, 
 we let 
\begin{equation}\label{generalconcatenate}
 0.a(1)a(2)\ldots 
\end{equation}
 denote the real value given by concatenating the digit expansions of consecutive entries of the aforementioned sequence. The first real value proved to 
 be normal is famously due to Champernowne \cite{Champernowne1933} and is given by the case whereby $a(n) = n$ for all $n \in \mathbb{N}$ in 
 \eqref{generalconcatenate}, in base 10. Shortly afterwards, Besicovitch \cite{Besicovitch1935} proved a result that may be applied to obtain the normality 
 of the corresponding constant for the $a(n) = n^2$ case \cite{PollackVandehey2015Besicovitch}. This normality result was then generalized by 
 Davenport and Erd\H{o}s \cite{DavenportErdos1952} for the case whereby $a(n)$ is a polynomial satisfying certain conditions. The polynomial cases 
 we have covered lead us to consider the behaviour of the digits in \eqref{generalconcatenate} for number-theoretic sequences. 

 A famous 1946 result due to Copeland and Erd\H{o}s \cite{CopelandErdos1946} provides the base-10 normality of \eqref{generalconcatenate} for the 
 case whereby $a(n) $ is equal to the $n^{\text{th}}$ prime number $p_{n}$. In this case, the constant 
\begin{equation}\label{numericalCE}
 0.235711131719232...
\end{equation}
 of the form indicated in \eqref{generalconcatenate} is referred to as the \emph{Copeland--Erd\H{o}s constant}. Copeland and Erd\H{o}s' 1946 article 
 \cite{CopelandErdos1946} is seminal within areas of number theory concerning normal numbers, and this inspires the exploration of variants of 
 \eqref{numericalCE}, with the use of number-theoretic sequences in place of $(p_{n} : n \in \mathbb{N} )$. As a natural variant of the 
 Copeland--Erd\H{o}s constant, we consider the constant 
\begin{equation}\label{numericalPCCE}
 0.012233444455666677888899999910101111... 
\end{equation}
 obtained from \eqref{generalconcatenate} by setting the sequence $(a(n) : n \in \mathbb{N})$ to be equal to the sequence $(\pi(n) : n \in 
 \mathbb{N}) $ given by the prime-counting function. Copeland and Erd\H{o}s' proof \cite{CopelandErdos1946} of the normality of 
 \eqref{numericalCE} relied on the property whereby the sequence $(p_{n} : n \in \mathbb{N})$ is strictly increasing and the property given by $p_{n} = 
 n^{1 + o(1)}$, but, since $(\pi(n) : n \in \mathbb{N} )$ is not strictly increasing, the techniques from \cite{CopelandErdos1946} cannot be translated so 
 as to be applicable to the constant in \eqref{numericalPCCE}. 

\section{Main construction}
 It seems that references on normal numbers related to Copeland and Erd\H{o}s' work in \cite{CopelandErdos1946}, including references such as 
 \cite{BaileyCrandall2002,BecherFigueiraPicchi2007,DavenportErdos1952,MadritschThuswaldnerTichy2008,NivenZuckerman1951} 
 that have inspired our work, have not involved the ``inverse'' constant in \eqref{numericalPCCE}. Moreover, integer sequences involving 
\begin{equation}\label{notinOEIS}
 (0, 1, 2, 2, 3, 3, 4, 4, 4, 4, 5, 5, 6, 6, 6, 6, 7, 7, 8, 8, 8, 8, 9, 9, 9, 9, 9, 9, 1, 0, \ldots) 
\end{equation}
 are not currently included in the On-Line Encyclopedia of Integer Sequences, where the tuple in \eqref{notinOEIS} is given by the consecutive digits of 
 the constant in \eqref{numericalPCCE}, which we refer to as the \emph{prime-counting Copeland--Erd\H{o}s (PCCE) constant}. In relation to this 
 constant, we are to apply the remarkable result that was originally formulated by Sz\"{u}sz and Volkmann in 1994 \cite{SzuszVolkmann1994} and that 
 was later corrected by Pollack and Vandehey \cite{PollackVandehey2015Besicovitch} and that is reproduced below, and we are to later explain 
 notation/terminology given in the following Theorem. 

\begin{theorem}\label{theoremSV}
 (Sz\"{u}sz $\&$ Volkmann, 1994) Suppose that $f$ is a differentiable function, and that $f$ is monotonically increasing and positive for all $x \geq n_{0}(f)$ 
 and that both $\eta(f)$ and $\eta(f')$ exist and that $0 < \eta(f) \leq 1$. It follows that $f$ is a Champernowne 
 function \cite{SzuszVolkmann1994} (cf.\ \cite{PollackVandehey2015Besicovitch}). 
\end{theorem}

\begin{definition}
 For a real-valued function $f$ defined on a domain containing $\mathbb{N}$ such that $f(n) > 0$, we let $\alpha(f) = \alpha_{g}(f)$ denote the real 
 number such that the base-$g$ expansion of this real number is of the form $0.b_{1}b_{2}...$, where $b_{n}$ denotes the base-$g$ expansion 
 of $ \left\lfloor f(n) \right\rfloor$ \cite{SzuszVolkmann1994}. 
\end{definition}

\begin{example}
 The PCCE constant in \eqref{numericalPCCE} is equal to $\alpha_{10}(\pi)$, writing $\pi$ in place of the prime-counting function. 
\end{example}

\begin{definition}
 A function $f$ such that $\lim_{x \to \infty} f(x) = \infty$ is a \emph{Champernowne function} if: For all $g \geq 2$, the value $\alpha_{g}(f)$ is normal 
 in base $g$ \cite{SzuszVolkmann1994}. 
\end{definition}

\begin{example}
 The identity function $f$ mapping $n$ to $n$ is a Champernowne function (cf.\ \cite{NakaiShiokawa1992}). 
\end{example}

\begin{definition}
 For a real-valued, positive function $f$, we set 
\begin{equation}\label{etadefinition}
 \eta(f) = \lim_{x \to \infty} \frac{\log f(x)}{\log x},
\end{equation}
 under the assumption that this limit exists. 
\end{definition}

\begin{example}\label{sqrtcons}
 As noted in \cite{PollackVandehey2015Besicovitch,SzuszVolkmann1994}, the constant $0.1112222233333334...$ given by setting $f(n) = \sqrt{n}$ in 
 $\alpha(f)$ is normal, with $\eta(f) = \frac{1}{2}$. 
\end{example}

 A difficulty associated with the explicit construction of a function $f$ that satisfies all of the required properties in Theorem \ref{theoremSV} and that can be 
 used, via Sz\"{u}sz and Volkmann's combinatorial method, to prove the normality of \eqref{numericalPCCE} has to do with the limit associated with 
 $\eta(f')$, since, for example, $f'$ cannot vanish infinitely often, in view of the definition in \eqref{etadefinition}. A more serious problem has to do 
 with how the distribution of groupings of digits in the PCCE constant would depend on the behaviour of the sequence $(g_{n} = p_{n+1} - p_{n} : n \in 
 \mathbb{N} )$ of prime gaps, since the number of times $m \in \mathbb{N}$ appears in $(\pi(n) : n \in \mathbb{N})$ is equal to $g_{m}$. In this 
 regard, we are to make use of a purported formula, under the assumption of Cram\'{e}r's conjecture, concerning the size of prime gaps. 

 \emph{Cram\'{e}r's conjecture} (cf.\ \cite{Cramer1936}) often refers to the purported estimate 
\begin{equation}\label{mainCramer}
 g_{n} = O(\log^2 p_{n}). 
\end{equation}
 The weaker formula 
\begin{equation}\label{needsRH}
 g_{n} = O(\sqrt{p_{n}} \log p_{n} ), 
\end{equation}
 which Cram\'{e}r proved \cite{Cramer1936} under the assumption of the Riemann Hypothesis (RH), is not sufficient for our purposes, as we later discuss. 
 The problem, here, has to do with how we would want $\lim_{n \to \infty} \frac{\log h_{n}}{\log n} $ to vanish for a function $h_{n}$ 
 approximating $g_{n} $. 

 Our construction of a function $f$ that satisfies the conditions in Theorem \ref{theoremSV} and that may be used to prove the normality of the PCCE 
 constant may be viewed as being analogous to how the $\Gamma$-function provides a differentiable analogue of the factorial function defined on natural 
 numbers. As indicated above, $\eta(f')$ could not be vanishing infinitely often, if we were to apply Sz\"{u}sz and Volkmann's combinatorial method 
 \cite{SzuszVolkmann1994}, which leads us to consider how $f'$ could be bounded, in such a way to guarantee the existence of the limit associated with 
 $ \eta(f')$. Informally, if we consider the graph of the prime-counting function as a function defined on $\mathbb{R}$, and if we consider a function $f$ 
 that projects onto the graph of $\pi$ under $\lfloor \cdot \rfloor$ and that has a derivative that is reasonably well behaved, we would expect the 
 derivative function $f'(x)$, for values $x$ such that $p_{m} \leq x < p_{m+1}$, to be ``reasonably close'' to $\frac{1}{g_{m}}$. Formalizing this 
 notion in a way that would allow us to apply Theorem \ref{theoremSV} is nontrivial, 
 as below. 

\begin{theorem}
 Under the assumption of Cram\'{e}r's conjecture, the PCCE constant is normal in base 10 and, more generally, the constant $0.a(1)a(2)\cdots$ is normal 
 in base $g$ for $a(n) = \pi(n)$. 
\end{theorem} 

\begin{proof}
 First suppose that the natural number $m \in \mathbb{N}_{\geq 2}$ is such that $\frac{1}{g_{m-1}} > \frac{1}{g_{m}}$. 
 Let $\varepsilon^{(m)} > 0$. We set $q_{1}^{(m)} 
 = p_{m} + \varepsilon^{(m)}$ for $\varepsilon^{(m)} < 1$, 
 and we set 
\begin{equation}\label{q2m}
 q_{2}^{(m)} = \frac{q_{1}^{(m)} (g_{m-1} - g_{m})}{g_{m-1} g_{m}^2} 
 + \frac{ g_{m} p_m - g_{m-1} p_m + g_{m-1} g_{m}}{g_{m-1} g_{m}^2}. 
 \end{equation}
 We may verify that \eqref{q2m} 
 reduces in such a way so that 
\begin{equation}\label{qreduced}
 q_{2}^{(m)} = 
 \frac{ \varepsilon^{(m)} \left( \frac{1}{g_{m}} - \frac{1}{g_{m-1}} \right) }{g_{m}} + \frac{1}{g_{m}}, 
\end{equation}
 which gives us that $q_{2}^{(m)} < \frac{1}{g_{m}}$. 
 By letting $\varepsilon^{(m)}$ be sufficiently small, with 
\begin{equation}\label{epsilonmsmall}
 0 < \varepsilon^{(m)} < \frac{1}{\frac{1}{g_{m-1}} - \frac{1}{g_{m}}}, 
\end{equation}
 the condition in \eqref{epsilonmsmall} gives us that 
 $0 < q_{2}^{(m)} < \frac{1}{g_{m}}$. 
 We define the function $h^{\varepsilon}_{m}(x)$ on the interval $[p_{m}, p_{m+1}]$ so that: 
 \[ h_{\varepsilon}^{(m)}(x) = \begin{cases} 
 \frac{1 - q_{2}^{(m)} g_{m-1}}{g_{m-1} \left(p_m - q_{1}^{(m)} \right)} x 
 + \frac{q_{2}^{(m)} g_{m-1} p_m - q_{1}^{(m)}}{g_{m-1} \left(p_m - q_{1}^{(m)} \right)} 
 & \text{if $p_{m} \leq x \leq p_{m} + \varepsilon^{(m)}$,} \\ 
 \frac{1 - q_{2}^{(m)} g_{m}}{ g_{m} \left(p_{m+1} - q_{1}^{(m) } \right)} 
 x + 
 \frac{q_{2}^{(m)} g_{m} p_{m+1} - q_{1}^{(m)}}{g_{m} \left(p_{m+1} - 
 q_{1}^{(m)} \right)} & \text{if $p_{m} + \varepsilon^{(m)} \leq x \leq p_{m+1}$}. 
 \end{cases} \] 

 Again for a natural number $m \in \mathbb{N}_{\geq 2}$, we proceed to suppose that $\frac{1}{g_{m-1}} < \frac{1}{g_{m}}$. We write $\delta^{(m)} 
 > 0$. We set $r_{1}^{(m)} = p_{m} + \delta^{(m)}$, with $\delta^{(m)} < 1$. We set 
\begin{equation}\label{setr2m}
 r_{2}^{(m)} = \frac{ r_{1}^{(m)} (g_{m-1} - g_{m})}{g_{m-1} g_{m}^2} + \frac{g_{m} p_m - g_{m-1} p_{m+1} + 2 g_{m-1} g_{m}}{g_{m-1} g_{m}^2}. 
\end{equation}
 We may verify that \eqref{setr2m} reduces so that 
\begin{equation}\label{r2reduce}
 r_{2}^{(m)} = \frac{ \delta^{(m)} \left( \frac{1}{g_{m}} - \frac{1}{g_{m-1}} \right) }{g_{m}} + \frac{1}{g_{m}}, 
\end{equation}
 so that $r_{2}^{(m)} > \frac{1}{g_{m}}$. 
 We define the function $h^{(m)}_{\delta}(x)$ on the interval $[p_{m}, p_{m+1}]$ so that: 
\[ h^{(m)}_{\delta}(x) = 
 \begin{cases} 
 \frac{1 - r_{2}^{(m)} g_{m-1}}{g_{m-1} \left(p_m - r_{1}^{(m)} \right)} x 
 + \frac{ r_{2}^{(m)} g_{m-1} p_m - r_{1}^{(m)} }{ g_{m-1} \left(p_m - r_{1}^{(m)} \right)} 
 & \text{if $p_{m} \leq x \leq p_{m} + \delta^{(m)}$}, \\ 
 \frac{1 - r_{2}^{(m)} g_{m} }{ g_{m} \left(p_{m+1} - r_{1}^{(m)} \right)} x + 
 \frac{ r_{2}^{(m)} g_{m} p_{m+1} - r_{1}^{(m)} }{ g_{m} \left(p_{m+1} - r_{1}^{(m)} 
 \right)} & \text{if $p_{m} + \delta^{(m)} \leq x \leq p_{m+1}$}. 
 \end{cases} \] 

 Finally, if $\frac{1}{g_{m-1}} = \frac{1}{g_{m}}$ for $m \in \mathbb{N}_{\geq 2}$, we define $h^{(m)}_{\text{null}}(x)$ on the interval $[p_{m}, 
 p_{m+1}]$ so that $h^{(m)}_{\text{null}}(x) = \frac{1}{g_{m}}$. 

 Now, set $f(x) = \frac{1}{96} \left(4 x^2+12 x+39\right)$ if $0 \leq x \leq \frac{3}{2}$, and set $f(x) = \frac{1}{4} \left(3 x^2-8 x+8\right)$ if 
 $\frac{3}{2} \leq x \leq 2$ and set $f(x) = x-1$ if $2 \leq x \leq 3$. Now, for $\pi(x) \in \mathbb{N}_{\geq 2}$, we set $ f(x) $ so that 
 \[ f(x) = \begin{cases} 
 \pi(x) + \int_{p_{\pi(x)}}^{x} h^{(\pi(x))}_{\varepsilon}(k) \, dk 
 & \text{if $\frac{1}{g_{\pi(x)-1}} > \frac{1}{g_{\pi(x)}}$}, \\ 
 \pi(x) + \int_{p_{\pi(x)}}^{x} h^{(\pi(x))}_{\delta}(k) \, dk 
 & \text{if $\frac{1}{g_{\pi(x)-1}} < \frac{1}{g_{\pi(x)}}$}, \\ 
 \pi(x) + \int_{p_{\pi(x)}}^{x} h^{(\pi(x))}_{\text{null}}(k) \, dk 
 & \text{if $\frac{1}{g_{\pi(x)-1}} = \frac{1}{g_{\pi(x)}}$}. 
 \end{cases} \] 
 By construction, we have that $f$ is a differentiable function and that $f$ is monotonically increasing and positive for $x \geq 0$. Moreover, the function 
 $f$ is constructed so that $\lfloor f(x) \rfloor = \pi(x)$ for all $x \geq 0$. We have that $\frac{1}{8} \leq f'(x) \leq 1$ for $0 \leq x \leq 3$, and, for $x 
 \geq 3$, we have that $f'(x)$ is either equal to $h^{(m)}_{\varepsilon}(x)$ or $h^{(m)}_{\delta}(x)$ or $h^{(m)}_{\text{null}}(x)$, according, 
 respectively, to the possibilities whereby $\frac{1}{g_{m-1}} > \frac{1}{g_{m}}$ and $\frac{1}{g_{m-1}} < \frac{1}{g_{m}}$ and $\frac{1}{g_{m-1}} 
 = \frac{1}{g_{m}}$. 

 If $\frac{1}{g_{m-1}} > \frac{1}{g_{m}}$, then, by construction, we have that 
\begin{equation}\label{inviewqred}
 0 < \frac{ \varepsilon^{(m)} \left( \frac{1}{g_{m}} - \frac{1}{g_{m-1}} \right) }{g_{m}} + \frac{1}{g_{m}} 
 \leq f'(x) \leq \frac{1}{g_{m-1}} 
\end{equation}
 for $p_{m} \leq x \leq p_{m+1}$, in view of \eqref{qreduced}, 
 where, for fixed $m$, the expression $\varepsilon^{(m)} > 0$ may be arbitrary, subject to \eqref{epsilonmsmall}. 

 If $\frac{1}{g_{m-1}} < \frac{1}{g_{m}}$, then, by construction, we have that 
\begin{equation}\label{recallepsilonupper}
 0 < \frac{1}{g_{m-1} } \leq f'(x) \leq 
 \frac{ \delta^{(m)} \left( \frac{1}{g_{m}} - \frac{1}{g_{m-1}} \right) }{g_{m}} + \frac{1}{g_{m}}, 
\end{equation}
 in view of \eqref{r2reduce}, and, for 
 $p_{m} \leq x \leq p_{m+1}$, the expression $\delta^{(m)} > 0$ is arbitrary on the specified interval. 

 Finally, by construction, if $\frac{1}{g_{m-1}} = \frac{1}{g_{m}}$, then 
 $$ 0 < \frac{1}{g_{m-1}} = \frac{1}{g_{m}} \leq f'(x) \leq \frac{1}{g_{m-1}} 
 = \frac{1}{g_{m}}. $$

 So, in any case, we find that 
\begin{equation*}
 \min\left\{ \overline{\varepsilon^{(\pi(x))}} 
 + \frac{1}{g_{\pi(x)}}, \frac{1}{g_{\pi(x)-1}} \right\} 
 \leq f'(x) \leq \max\left\{ \overline{\delta^{(\pi(x))}} 
 + \frac{1}{g_{\pi(x)}}, \frac{1}{g_{\pi(x)-1}} \right\}, 
\end{equation*}
 where the ``overlined'' expressions given above 
 are, respectively, given by the positive terms involving $\varepsilon^{(m)}$ and $\delta^{(m)}$ 
 in \eqref{inviewqred} and \eqref{recallepsilonupper}, writing $m = \pi(x)$. 

 Under the assumption of Cram\'{e}r's conjecture, as formulated in \eqref{mainCramer}, we find that there exists $M > 0$ and $x_{0} \in \mathbb{R}$ 
 such that 
\begin{equation}\label{underCramer}
 \forall y \geq x_{0} \ \frac{1}{ M \log^{2} p_{y} } \leq \frac{1}{g_{y}} \leq \frac{1}{2}. 
\end{equation}
 So, for suitable natural numbers $m \in \mathbb{N}_{\geq 2}$, 
 by taking $\overline{\varepsilon^{(m)}}$ to be sufficiently small 
 throughout a given interval $[p_{m}, p_{m+1}]$, and similarly for $\overline{\delta^{(m)}}$, from 
 the consequence of 
 Cram\'{e}r's conjecture given in \eqref{underCramer}, we may deduce that 
\begin{equation}\label{againCramer}
 \frac{1}{M \log^{2} p_{\pi(x)} } \leq f'(x) \leq 0.51, 
\end{equation}
 for sufficiently large $x$. 
 Using explicit bounds as in \cite{Dusart1999,Dusart2010,Rosser1941} for the prime-counting function and for the sequence 
 of primes, we may obtain, from \eqref{againCramer}, that 
\begin{equation}\label{2707273707972717870707P7M1A}
 \frac{1}{ M \log ^2\left(\frac{x \log \left(\frac{x \log \left(\frac{x}{\log (x)-1.1}\right)}{\log (x)-1.1}\right)}{\log (x)-1.1}\right)} 
 \leq f'(x) \leq 0.51, 
\end{equation}
 for sufficiently large $x$. By taking the natural logarithm of $f'(x)$ and the bounds in 
 \eqref{2707273707972717870707P7M1A}, and then dividing by $\log(x)$, 
 it is a matter of routine to verify that the limit as $x \to \infty$ of the resultant 
 bounds vanishes, so that $\eta(f')$ exists, as desired, with $\eta(f') = 0$. 

 By construction, we have that $ \pi(x) \leq f(x) \leq \pi(x) + 1 $ 
 for all $x$. So, we find that 
\begin{equation}\label{209293909992919785777P7M1A}
 \frac{\log \left(\frac{x}{\log (x)-1}\right)}{\log (x)} \leq 
 \frac{ \log(f(x)) }{\log(x)} \leq \frac{\log \left(\frac{x}{\log (x)-1.1}+1\right)}{\log (x)}. 
\end{equation}
 By taking the limit as $x \to \infty$ of the upper and lower bounds given in 
 \eqref{209293909992919785777P7M1A}, this gives us the same value of $1$, 
 so that $\eta(f) = 1$, as desired. 
 
 So, we have that $f$ is a differentiable function and is monotonically increasing and positive, 
 and we have that $\eta(f)$ and $\eta(f')$ exist and that 
 $0 < \eta(f) \leq 1$, 
 and we have that $\lim_{x \to \infty} f(x) = \infty$. 
 So, from Theorem \ref{theoremSV}, 
 we have that $f$ is a Champernowne function. 
\end{proof}

\section{Discussion}
 Following the work of Sz\"{u}sz and Volkmann, 
 as in the main article that has inspired our work \cite{SzuszVolkmann1994}, 
 we adopt the convention whereby 
 strings of digits are counted ``with overlaps allowed''
 in the sense that two occurrences of the same substring in a larger string
 are counted separately if these strings happen to overlap or otherwise. 
 For example, Sz\"{u}sz and Volkmann \cite{SzuszVolkmann1994} 
 provide the illustration whereby the substring $131$ is counted 
 four times within $713131051310131$, noting the substring $131$ overlapping with itself
 within $13131$. This convention concerning normal numbers 
 does not agree with the Mathematica commands {\tt StringCount} or {\tt SequenceCount}, 
 but, as a way of avoiding this kind of issue, we may instead use the Wolfram {\tt StringPosition}
 command. For example, inputting 
\begin{verbatim}
StringPosition["713131051310131", "131"]
\end{verbatim}
 into the Mathematica Computer Algebra System, we obtain a list of four elements, 
 which agrees with Sz\"{u}sz and Volkmann's convention for eunmerating 
 subsequences of digits. 
 
 By taking the first 10 million entries of the integer sequence 
 $ ( \pi(n) : n \in \mathbb{N}_{0})$, and then converting this subsequence into a string
 without commas or spaces or brackets, 
 and then counting the number of occurrences of 
 each digit among $1$, $2$, $\ldots$, $9$, $0$ 
 using the Wolfram {\tt StringPosition} function, and then 
 computing the frequency by dividing by the length of the string
 corresponding to $ ( \pi(n) : n \in \mathbb{N}_{\leq 10,000,000} )$, 
 we obtain the frequencies listed in Table \ref{TablePCCE}. 

\begin{table}[t]
\centering

 \begin{tabular}{ | c | c | }
 \hline
 Digit & Frequency corresponding to $ ( \pi(n) )_{n \in \mathbb{N}_{\leq 10^7}}$ \\ \hline 
 1 & 0.110875 \\ \hline 
 2 & 0.111823 \\ \hline 
 3 & 0.112635 \\ \hline 
 4 & 0.113276 \\ \hline 
 5 & 0.113596 \\ \hline 
 6 & 0.102678 \\ \hline 
 7 & 0.0835609 \\ \hline 
 8 & 0.0836607 \\ \hline 
 9 & 0.0839711 \\ \hline 
 0 & 0.0839241 \\ 
 \hline
 \end{tabular}

 \ 

\caption{Numerical evidence of the simple normality of the PCCE constant.}\label{TablePCCE}
\end{table}

 For each of the given frequencies $f$ shown in Table \ref{TablePCCE}, we have that $\left| f - \frac{1}{10} \right| < 0.0165$. The data in 
 Table \ref{TablePCCE} 
 may suggest that lower-value digits, according to the given ordering, may appear more frequently within the PCCE constant, say, in the sense 
 that the frequencies, within substrings corresponding to $(\pi(n) : n \in \mathbb{N}_{0})$, for lower-order digits may be, in general, higher 
 compared to the frequencies for higher-valued digits. 
 This recalls the statistical principle 
 known as \emph{Benford's law}, and we encourage the number-theoretic exploration of 
 this phenomenon with the use of the PCCE constant. 

 One may wonder why it would be appropriate, for the purposes of our application of the combinatorial 
 method due to Sz\"{u}sz and Volkmann \cite{SzuszVolkmann1994}, 
 to use the formulation of Cram\'{e}r's conjecture in \eqref{mainCramer}, 
 as opposed to the estimate for prime gaps that Cram\'{e}r proved under the assumption of the RH. 
 The problem, here, has to do with the lower bound that would correspond to \eqref{underCramer}, 
 if estimates other than Cram\'{e}r's conjecture were to be used. 
 If one were to attempt to make use of a lower bound as in 
\begin{equation}\label{attemptRH}
 \frac{1}{M \sqrt{p_{\pi(x)}} \log p_{\pi(x)} } \leq f'(x) \leq 0.51, 
\end{equation}
 under the assumption of the RH, in the hope of applying 
 the estimate of prime gaps shown in \eqref{needsRH}, 
 we encounter a problem concerning the behaviour of prime gaps, in the sense described as foillows. 

 To apply Theorem \ref{theoremSV}, the limit corresponding to 
 $\eta(f')$ would have to exist, but, by manipulating the inequalities in \eqref{attemptRH} so that 
\begin{equation}\label{showRHnotenough}
 \frac{ \log\left( \frac{1}{M \sqrt{p_{\pi(x)}} \log p_{\pi(x)} } \right) }{\log(x)} \leq 
 \frac{ \log\left( f'(x) \right) }{\log(x)} \leq \frac{ \log\left( 0.51\right)}{\log(x)}, 
\end{equation}
 we encounter the problem indicated as follows. By taking the limit as $x \to \infty$
 for the bounds in \eqref{showRHnotenough}, 
 the Prime Number Theorem gives us that 
 the limit corresponding to the lower bound reduces to $-\frac{1}{2}$, 
 whereas the limit corresponding to the upper bound vanishes, 
 so the valuations for these limits corresponding to the upper and lower bounds for 
 \eqref{showRHnotenough} would not imply that $\eta(f')$ 
 would exist. 
 How could the RH be used, in place of the 
 formulation of Cram\'{e}r's conjecture in \eqref{mainCramer}, to prove the normality of the PCCE constant? 

\begin{table}[t]
\centering

 \begin{tabular}{ | c | c | }
 \hline
 Digit & Frequency corresponding to $ \left( \left\lfloor \sqrt{n} \right\rfloor \right)_{n \in \mathbb{N}_{\leq 10^7}}$ \\ \hline 
 1 & 0.156047 \\ \hline 
 2 & 0.198942 \\ \hline 
 3 & 0.0980257 \\ \hline 

 4 & 0.0740972 \\ \hline 
 5 & 0.0758164 \\ \hline 
 6 & 0.0762815 \\ \hline 
 7 & 0.0776263 \\ \hline 
 8 & 0.0793403 \\ \hline 
 9 & 0.0810544 \\ \hline 
 0 & 0.0827685 \\ 
 \hline
 \end{tabular}

 \ 

\caption{Frequencies associated with the appearance of digits in the 
 Sz\"{u}sz--Volkmann constant (cf.\ \cite{PollackVandehey2015Besicovitch,SzuszVolkmann1994}) 
 indicated in Example \ref{sqrtcons}.}\label{sqrttable}
\end{table}

 The data in Table \ref{TablePCCE} may be regarded as offering strong evidence for the normality of the PCCE constant if we compare these data 
 to the correpsonding frequencies for the Sz\"{u}sz--Volkmann constant (cf.\ \cite{PollackVandehey2015Besicovitch,SzuszVolkmann1994}) given by 
 concatenating the digit expansions of consecutive integer parts of the 
 square root function, as in Example \ref{sqrtcons}. 
 Given the notoriously unpredictable behaviour of prime gaps, one might think that 
 the frequencies of digits in the PCCE constant 
 would be similarly unpredictable, compared to the 
 Sz\"{u}sz--Volkmann constant, but the data in Tables 
 \ref{TablePCCE} and \ref{sqrttable}
 suggest that the PCCE constant is actually much more well behaved compared to 
 the $a(n) = \left\lfloor \sqrt{n} \right\rfloor$ case of \eqref{generalconcatenate}. 
 For example, the largest frequency in Table \ref{sqrttable} is much farther 
 away from the desire mean of $0.1$, relative to the PCCE constant, 
 and similarly for the smallest frequency in Table \ref{sqrttable}. 

\subsection*{Acknowledgements} 
 The author was supported through a Killam Postdoctoral Fellowship from the Killam Trusts. The author is thankful to Joel E.\ Cohen for useful comments 
 concerning the subject of this article, and the author is thankful to Karl Dilcher for useful discussions concerning this article.

 \

John M.\ Campbell

Department of Mathematics and Statistics 

Dalhousie University

{\tt jmaxwellcampbell@gmail.com}

\end{document}